\documentclass[12pt,leqno,amscd, amsfonts, amssymb,pstricks,verbatim]{amsart}
\usepackage{amsfonts}
\usepackage{amscd}
\usepackage{mathrsfs}
\usepackage{amsmath}
\usepackage{amssymb}
\usepackage{hyperref}
\usepackage{bookmark}

\def\a{\alpha}
\def\r{\gamma}
\def\b{\beta}

\def\Z{\mathbb{Z}}
\def\N{\mathbb{N}}

\def\C{\mathbb{C}}

\numberwithin{equation}{section}
\newtheorem{theo}{Theorem}[section]
\newtheorem{defi}[theo]{Definition}
\newtheorem{coro}[theo]{Corollary}
\newtheorem{lemm}[theo]{Lemma}
\newtheorem{prop}[theo]{Proposition}
\newtheorem{remm}[theo]{Remark}
\newtheorem{clai}{Claim}

\allowdisplaybreaks

\begin{document}

\title[Irreducible tensor product modules]{Irreducible tensor product modules over the affine-Virasoro algebra of type $A_1$}

\author{Qiu-Fan Chen, Yu-Feng Yao}

\address{Department of Mathematics, Shanghai Maritime University,
 Shanghai, 201306, China.}\email{chenqf@shmtu.edu.cn}
\address{Department of Mathematics, Shanghai Maritime University,
 Shanghai, 201306, China.}\email{yfyao@shmtu.edu.cn}



\thanks{This work is supported by National Natural Science Foundation of China (Grant Nos. 11801363, 11771279, 12071136 and 11671247).}

\begin{abstract}
In this paper, we construct a class of non-weight modules over the affine-Virasoro algebra of type $A_1$ by taking tensor products of irreducibles  defined in \cite{CH1} with irreducible highest weight modules. The irreducibility and the isomorphism classes of these modules are determined. Moreover, we show that these tensor product modules are different from the known non-weight modules. Finally, we realize some tensor product modules as induced modules from modules over certain subalgebras of the affine-Virasoro algebra of type $A_1$, and give sufficient and necessary conditions for these induced modules to be reducible.
\end{abstract}

\maketitle

\section{Introduction}
Throughout the paper, we denote by $\C ,\,\Z,\,\C^*,\,\Z_+,\,\N$ the sets of complex numbers, integers, nonzero complex numbers, nonnegative integers and positive integers, respectively. All algebras (modules, vector spaces) are assumed to be  over $\C$. For a Lie algebra $\mathfrak{g}$, we use $U(\mathfrak{g})$ to denote the universal enveloping algebra of $\mathfrak{g}$.  More generally, for a subset $X$ of $\mathfrak{g}$, we use $U(X)$ to denote the universal enveloping algebra of the subalgebra of $\mathfrak{g}$ generated by $X$.

It is well known that representation theory of the Virasoro algebra and affine Kac Moody Lie algebras plays an important role both in physics and in mathematics. The Virasoro algebra acts on any (except when the level is negative the dual coxeter number) highest weight module of the affine Lie algebra through the use of the famous Sugawara operators.  The affine Lie algebras admit representations on the Fock space and hence admit representations of the Virasoro algebra. This close
relationship strongly suggests that they should be considered simultaneously, i.e., as one algebraic structure, and hence has led to the definition of the so-called affine-Virasoro algebra \cite{K,K1}. Highest weight representations and integrable  representations of the affine-Virasoro algebras have been  extensively studied  (cf.~\cite{B}, \cite{EJ}, \cite{JY}-\cite{LQ},  \cite{XH}). Quite recently, the authors  gave the classification of irreducible quasi-finite modules  over the affine-Virasoro algebras in \cite{LPX}. The affine-Virasoro algebras  are very meaningful in the sense that they are closely connected to the conformal field theory. For example, the even part of the $N=3$ superconformal algebra \cite{CL} is just the affine-Virasoro algebra of type $A_1$. The affine-Virasoro algebra of type $A_1$, denoted by $\mathfrak L$,  is defined as the Lie algebra with $\C$-basis $\{e_i,\, f_i,\,h_i,\, d_i,\,C\mid i\in\Z\}$ subject to the following Lie brackets:
\begin{equation*}\label{L-action}
\aligned
&[e_i,f_j]=h_{i+j}+i\delta_{i+j,0}C,\\
&[h_i,e_j]=2e_{i+j},\quad [h_i,f_j]=-2f_{i+j},\\
&[d_i,d_j]=(j-i)d_{i+j}+\delta_{i+j,0}\frac{i^3-i}{12}C,\\
&[d_i,h_j]=jh_{i+j},\quad [h_i,h_j]=-2i\delta_{i+j,0}C,\\
&[d_i,e_j]=je_{i+j},\quad [d_i,f_j]=jf_{i+j},\\
&[e_i,e_j]=[f_i,f_j]=[C,\mathfrak{L}]=0.
\endaligned
\end{equation*}
It is clear that  $\mathfrak{h}:=\C d_0+\C h_0+\C C$ is the Cartan subalgebra  of $\mathfrak{L}$. Moreover, $\mathfrak L$ admits a triangular decomposition:
$$\mathfrak{L}=\mathfrak{L}_{-}\oplus\mathfrak{h}\oplus\mathfrak{L}_{+},$$
where
\begin{eqnarray*}
\mathfrak{L}_{-}={\rm span\,}_{\C}\{e_{-i}, f_{-i},  h_{-i}, d_{-i}, f_0 \mid i\in\N\}
\end{eqnarray*}
and \begin{eqnarray*}\mathfrak{L}_{+}={\rm span\,}_{\C}\{e_{i}, f_{i},  h_{i}, d_{i}, e_0\mid i\in\N\}.
\end{eqnarray*}
The classification of all irreducible Harish-Chandra modules over $\mathfrak L$ was achieved in \cite{GHL}.

In recent years, many authors constructed various irreducible non-Harish-Chandra modules and irreducible non-weight modules (cf.~\cite{BM}, \cite{CTZ, CZ}, \cite{HCS}, \cite{MW}-\cite{TZ}). In particular, J. Nilsson \cite{N} constructed  a class of  $\mathfrak{sl_{n+1}}$-modules that are free of rank one when restricted to the Cartan subalgebra.  Since then, this kind of non-weight modules, which many authors call $U(\mathfrak{h})$-free modules, have been extensively studied. Especially, the authors classified the $U(\mathfrak{h})$-free modules of rank one for  $\mathfrak{L}$ in \cite{CH1}. Moreover, the irreducibility and isomorphism classes of these modules were determined therein.

It is well known that an important way to construct new modules over an algebra is to consider the linear tensor product of two known modules
over the algebra (cf.~ \cite{CG2, CHSY}, \cite{GLW}, \cite{TZ1}, \cite{Z}). The purpose of the present paper is to construct new irreducible non-weight $\mathfrak{L}$-modules by taking tensor product of irreducible modules defined in \cite{CH1} with irreducible highest weight modules.

The present paper is organized as follows. In Section 2, we recall some known modules and results from \cite{CH1}.  Section 3 is devoted to studying the irreducibility of the tensor product modules $M(\lambda, \alpha, \beta, \gamma)\otimes V(\theta, \epsilon, \eta)$, where $M(\lambda, \alpha, \beta, \gamma)=\Omega(\lambda,\a, \b, \r), \Delta(\lambda, \a, \b, \r)$ or  $\Theta(\lambda, \a, \b, \r)\, (2\b \notin \Z_+)$ (see \S\,\ref{pre} for the definition). In Section 4, we give a necessary and sufficient condition for two irreducible tensor modules to be isomorphic. In Section 5 we show that this class of tensor product modules are new by showing that they are not isomorphic to the known ones.  In the final section, we realize some tensor product modules as induced modules from modules over certain subalgebras of $\mathfrak{L}$, and give sufficient and necessary conditions for these induced modules to be reducible.

\section{Preliminaries}\label{pre}
Let us first recall the definitions of the  $\mathfrak{L}$-modules $\Omega(\lambda,\a, \b, \r), \Delta(\lambda, \a, \b, \r), \Theta(\lambda, \a, \b, \r)$ and $V(\eta, \epsilon, \theta)$ concerned in this paper, and some basic properties of them. Denote by $\C[s,t]$ the polynomial algebra in variables $s$ and $t$ with coefficients in $\C$. As vector spaces, $\Omega(\lambda,\a, \b, \r), \Delta(\lambda, \a, \b, \r)$ and $\Theta(\lambda, \a, \b, \r)$ coincide with $\C[s,t]$.

\begin{defi}\label{defi2.2} For $\lambda,\a\in\C^*, \b,\r\in\C, i\in\Z$ and $g(s,t)\in\C[s,t]$, define the $\mathfrak{L}$-module action on $\C[s,t]$ as follows:
\begin{align*}
\Omega(\lambda,\a, \b, \r):&\ \ \ \ e_i\cdot g(s,t)=\lambda^i\a g(s-i,t-2),\\
& \ \ \ \ f_i\cdot g(s,t)=-\frac{\lambda^i}\a(\frac{t}{2}-\b)(\frac{t}{2}+\b+1)g(s-i,t+2),\\
& \ \ \ \ h_i\cdot g(s,t)=\lambda^itg(s-i,t),\ \ \ \ d_i\cdot g(s,t)=\lambda^i(s+i\r)g(s-i,t),\\
& \ \ \ \ C\cdot g(s,t)=0;\\
\Delta(\lambda, \a, \b, \r):&\ \ \ \ e_i\cdot g(s,t)=-\frac{\lambda^i}{\a}(\frac{t}{2}+\b)(\frac{t}{2}-\b-1)g(s-i,t-2),\\
& \ \ \ \ f_i\cdot g(s,t)=\lambda^i\a g(s-i,t+2), \ \ \ \ h_i\cdot g(s,t)=\lambda^{i}tg(s-i,t),\\
& \ \ \ \ d_i\cdot g(s,t)=\lambda^i(s+i\r)g(s-i,t),\\
& \ \ \ \ C\cdot g(s,t)=0;\\
\Theta(\lambda, \a, \b, \r):&\ \ \ \ e_i\cdot g(s,t)=\lambda^i\a(\frac{t}{2}+\b)g(s-i,t-2),\\
& \ \ \ \ f_i\cdot g(s,t)=-\frac{\lambda^i}{\a}(\frac{t}{2}-\b)g(s-i,t+2),\\
& \ \ \ \ h_i\cdot g(s,t)={\lambda^i}tg(s-i,t),\ \ \ \ d_i\cdot g(s,t)=\lambda^i(s+i\r)g(s-i,t),\\
& \ \ \ \ C\cdot g(s,t)=0.
\end{align*}
\end{defi}
It is worthwhile to point out that $\C[s,t]$ in each case has the same module structure over the subalgebra $\text{span}_{\C}\{h_i,d_i,C\mid i\in\Z\}$.
%
%
For later use, we need the following known result on conditions for irreducibility and a classification of isomorphism classes for the modules constructed above.
\begin{prop}\label{pop1}(cf. \cite{CH1})
Keep notations as above, then the following statements hold.
\begin{itemize}
\item[(1)]
$\Omega(\lambda,\a, \b, \r)$ and $\Delta(\lambda, \a, \b, \r)$ are irreducible for any $\lambda,\a\in\C^*$ and $\b,\r\in\C;$ While $\Theta(\lambda, \a, \b, \r)$ is irreducible if and only if $2\b \notin \Z_+$.
\item[(2)]
Let $\lambda_1, \lambda_2, \a_1,\a_2\in\mathbb{C}^*,\b_1, \b_2,\r_1, \r_2\in\mathbb{C}$. Then
$$\Omega(\lambda_1,\a_1,\b_1,\r_1), \Delta(\lambda_1,\a_1,\b_1,\r_1), \Theta(\lambda_1,\a_1,\b_1,\r_1)$$ are pairwise non-isomorphic for all parameter choices. Moreover,
\begin{eqnarray*} \label{xxit11}&\Omega(\lambda_1,\a_1,\b_1,\r_1)\cong
\Omega(\lambda_2, \a_2,\b_2,\r_2)\Longleftrightarrow & (\lambda_1,\a_1,\b_1,\r_1)=(\lambda_2,\a_2,\b_2,\r_2)\\
&&{\rm or}\,\, (\lambda_1,\a_1,\b_1,\r_1)=(\lambda_2,\a_2,-\b_{2}-1,\r_2);\nonumber\\
\label{xxit1}&\Delta(\lambda_1,\a_1,\b_1,\r_1)\cong
\Delta(\lambda_2,\a_2,\b_2,\r_2)\Longleftrightarrow & (\lambda_1,\a_1,\b_1,\r_1)=(\lambda_2,\a_2,\b_2,\r_2)\\
&&{\rm or}\,\, (\lambda_1,\a_1,\b_1,\r_1)=(\lambda_2,\a_2,-\b_{2}-1,\r_2);\nonumber\\
\label{xxit2}&\Theta(\lambda_1,\a_1,\b_1,\r_1)\cong\Theta(\lambda_2,\a_2,\b_2,\r_2)\Longleftrightarrow & (\lambda_1,\a_1,\b_1,\r_1)=(\lambda_2,\a_2,\b_2,\r_2).\end{eqnarray*}
\end{itemize}
\end{prop}
For any $\eta, \epsilon, \theta\in\C$, let $I(\eta, \epsilon, \theta)$ be the left ideal of $U(\mathfrak{L})$ generated by the following elements
$$\{e_0, e_{i}, f_{i},  h_{i}, d_{i}\mid i\in\N\}\cup\{d_0-\eta, h_0-\epsilon, C-\theta\}. $$
The Verma $\mathfrak{L}$-module with highest weight $(\eta,  \epsilon, \theta)$ is defined as the quotient module
\begin{equation*}\overline{V}(\eta, \epsilon, \theta)=U(\mathfrak{L})/I(\eta, \epsilon, \theta).\end{equation*}
By the PBW theorem, $\overline{V}(\eta, \epsilon, \theta)$ has a basis  consisting of all vectors of the form
\begin{equation*}f_{-q}^{F_{-q}}\cdots f_0^{F_0}e_{-p}^{E_{-p}}\cdots e_{-1}^{E_{-1}}h_{-m}^{H_{-m}}\cdots h_{-1}^{H_{-1}}d_{-n}^{D_{-n}}\cdots d_{-1}^{D_{-1}}\cdot v_h ,\end{equation*}
where $v_h$ is the coset of $1$ in $\overline{V}(\eta, \epsilon, \theta)$, and $$D_{-1},\cdots,D_{-n}, H_{-1}, \cdots, H_{-m}, E_{-1}, \cdots, E_{-p}, F_0,  \cdots, F_{-q}\in\Z_+.$$
Then we have the irreducible  highest weight module $V(\eta, \epsilon, \theta)=\overline{V}(\eta, \epsilon, \theta)/J$, where $J$ is the unique maximal proper submodule of $\overline{V}(\eta, \epsilon, \theta)$. Readers can refer to \cite{B, K} for the structure of $V(\eta, \epsilon, \theta)$.

In the rest of this paper, we will always assume $\lambda,\a\in\C^*, \b,\r,\eta, \epsilon, \theta \in\C$, $M(\lambda, \alpha, \beta, \gamma)=\Omega(\lambda, \alpha, \beta, \gamma)$, $\Delta(\lambda, \alpha, \beta, \gamma)$ or $\Theta(\lambda, \alpha, \beta, \gamma)$ constructed in Definition \ref{defi2.2}, and  $V(\eta, \epsilon, \theta)$ is an  irreducible  highest weight  $\mathfrak{L}$-module. Take a tensor product of the non-weight module $M(\lambda, \alpha, \beta, \gamma)$ with  irreducible  highest weight module $V(\eta, \epsilon, \theta)$. Clearly, the tensor product $\mathfrak{L}$-module $M(\lambda, \alpha, \beta, \gamma)\otimes V(\eta, \epsilon, \theta)$ is not a weight module.
\section{Irreducibility of the tensor product modules}
In this section, we show the irreducibility of the tensor product $\mathfrak{L}$-module $M(\lambda, \alpha, \beta, \gamma)\otimes V(\eta, \epsilon, \theta)$, where $2\b \notin \Z_+$ when $M=\Theta$.
\begin{theo}\label{theoo1}
The tensor product module $M(\lambda, \alpha, \beta, \gamma)\otimes V(\eta, \epsilon, \theta)$ is irreducible provided  that $2\b \notin \Z_+$ when $M=\Theta$.
\end{theo}
\begin{proof}
Let $W$ be a nonzero submodule of $M(\lambda, \alpha, \beta, \gamma)\otimes V(\eta, \epsilon, \theta)$. We need to show that $W=M(\lambda, \alpha, \beta, \gamma)\otimes V(\eta, \epsilon, \theta)$.  It is clear that, for any $v\in V(\eta, \epsilon, \theta)$, there exists a positive integer $K(v)$ such that $d_m\cdot v=e_m\cdot v=f_m\cdot v=h_m\cdot v=0$ for all $m\geq K(v)$.
Take any nonzero element $w=\sum_{i=0}^{r}a_i(t)s^{i}\otimes v_i\in W$ with $a_i(t)\in\C[t], v_i\in V(\eta, \epsilon, \theta), a_{r}(t)\neq0, v_r\neq0$ and $r\in \Z_+$ is minimal.
\begin{clai}\label{claim1}
$r=0$.
\end{clai}
Let $K={\rm max }\{K(v_i)\mid i=0, 1, \cdots, r\}$. Then we have
\begin{equation}\label{eq-1}
\lambda^{-m}d_m\cdot w=\sum_{i=0}^{r}(s+m\gamma)a_{i}(t)(s-m)^i\otimes v_i\in W, \,\,\,\forall m\geq K.
\end{equation}

Case (i): $\gamma\neq0$.

In this case, we write the right hand side of (\ref{eq-1}) as
\begin{equation*}
\sum_{i=0}^{r+1}m^iw_i\in W,  \,\,\,\forall m\geq K,
\end{equation*}
where all $w_i\in M(\lambda, \alpha, \beta, \gamma)\otimes V(\eta, \epsilon, \theta)$ are independent of $m$. Taking $m=K, K+1, \cdots, K+r+1$, we see that the coefficient matrix of $w_i$ is a Vandermonde matrix. So each $w_i\in W$. Especially,  $w_{r+1}=(-1)^{
r}\gamma a_{r}(t)\otimes v_{r}\in W$. Thus, $r$ must be zero by its minimality.

Case (ii): $\gamma=0$.

Similar arguments as in case (i) yield that $w_{r}= s a_{r}(t)\otimes v_{r}\in W$. Moreover, when $M(\lambda, \alpha, \beta, \gamma)=\Omega(\lambda,\a, \b, \r)$, we have
\begin{equation*}
\alpha^{-1}(\lambda^{-K}e_{K}-\lambda^{-K-1}e_{K+1})\cdot w_{r}= (-1)^{r}a_{r}(t-2)\otimes v_{r}\in W.
\end{equation*}
When $M(\lambda, \alpha, \beta, \gamma)=\Delta(\lambda, \a, \b, \r)$ or $\Theta(\lambda, \a, \b, \r)$, it follows from similar computation that
$a_{r}(t+2)\otimes v_{r}\in W$, $(\frac{t}{2}+\beta)a_{r}(t-2)\otimes v_{r}\in W$, respectively. In conclusion, $r$ must be zero in each situation by its minimality.
\begin{clai}
$W=M(\lambda, \alpha, \beta, \gamma)\otimes V(\eta, \epsilon, \theta).$
\end{clai}
By Claim \ref{claim1}, $a_0(t)\otimes v_0\in W$. Fix this $v_0$ and let
$$P=\{a(s,t)\in\C[s,t]\mid a(s,t)\otimes v_0\in W\}.$$
For any $a(t)\in \C[t], k\in \N$ and $m\geq K(v_0)$, by induction on $k$, we  have the following two formulae.
\begin{eqnarray}
\label{hm1}\lambda^{-mk}h_m^{k}\cdot \big(a(t)\otimes v_0\big)=t^{k}a(t)\otimes v_0, \\
\label{hm2}\lambda^{-mk}d_m^{k}\cdot \big(a(t)\otimes v_0\big)=\prod_{i=0}^{k-1}(s+m\gamma-mi)a(t)\otimes v_0.
\end{eqnarray}
We write $W$ as  $W_\Omega, W_\Delta$ and $W_\Theta$ to emphasis that $W$ is an $\mathfrak L$-submodule of $\Omega(\lambda,\a, \b, \r)\otimes V(\eta, \epsilon, \theta),$ $\Delta(\lambda, \a, \b, \r)\otimes V(\eta, \epsilon, \theta)$ and $\Theta(\lambda, \a, \b, \r)\otimes V(\eta, \epsilon, \theta)$, respectively. By Definition \ref{defi2.2}, one can inductively show that for $m\geq K(v_0)$
\begin{eqnarray*}\label{dda-a-1}\lambda^{-mk}\a^{-k}e_m^k\cdot \big(a_0(t)\otimes v_0\big)&=&a_0(t-2k)\otimes v_0\in W_\Omega,\\
\label{dda-a-2} \lambda^{-mk}\a^{-k}f_m^k\cdot \big(a_0(t)\otimes v_0\big)&=&a_0(t+2k)\otimes v_0\in W_\Delta,\\
\label{eqn5}\lambda^{-mk}\a^{-k}e_m^k\cdot \big(a_0(t)\otimes v_0\big)&=&\prod_{n=0}^{k-1}(\frac{t}{2}+\b-n)a_0(t-2k)\otimes v_0\in W_\Theta,\\
\label{eqn6}\lambda^{-mk}(-\a)^{-k}f_m^k\cdot  \big(a_0(t)\otimes v_0\big)&=&\prod_{n=0}^{k-1}(\frac{t}{2}-\b+n)a_0(t+2k)\otimes v_0\in W_\Theta.\end{eqnarray*}
Note that we can choose $k$ large enough so that $$\big(a_0(t),a_0(t-2k)\big)=1,\ \big(a_0(t),a_0(t+2k)\big)=1$$  and  $$ \big(\prod_{n=0}^{k-1}(\frac{t}{2}+\b-n)a_0(t-2k), \prod_{n=0}^{k-1}(\frac{t}{2}-\b+n)a_0(t+2k)\big)=1,\quad {\rm if }\ 2\b\notin\Z_+.$$
It follows from these and \eqref{hm1} that
$$1\otimes v_0\in W_\Omega, \,1\otimes v_0\in W_\Delta  \quad {\rm and }\ 1\otimes v_0\in W_\Theta,$$ which in turn force $\C[t]\subseteq P$ in each case. Since \eqref{hm2} implies that $P$ is stable under the multiplication by $s$, it follows that $P=\C[s, t]= M(\lambda, \alpha, \beta, \gamma)$. Hence, $M(\lambda, \alpha, \beta, \gamma)\otimes U(\mathfrak L)v_0\subseteq W$. Consequently, $W=M(\lambda, \alpha, \beta, \gamma)\otimes V(\eta, \epsilon, \theta)$ by the irreducibility of $V(\eta, \epsilon, \theta)$. We complete the proof.
\end{proof}

\section{Isomorphism classes of the tensor product modules}
In this section, we give the classification of the isomorphism classes of the tensor product $\mathfrak{L}$-modules $M(\lambda, \alpha, \beta, \gamma)\otimes V(\eta, \epsilon, \theta)$, where $2\b \notin \Z_+$ when $M=\Theta$.
\begin{theo}\label{theoo2}
Let $\lambda_i, \a_i\in\mathbb{C}^*,\b_i,\r_i, \theta_i, \epsilon_i, \eta_i\in\mathbb{C}$, where $i=1, 2$, and $V(\eta_1, \epsilon_1,\theta_1)$, $V(\eta_2, \epsilon_2, \theta_2)$ are irreducible  highest weight modules. Let $M(\lambda, \alpha, \beta, \gamma)=\Omega(\lambda,\a, \b, \r)$, $\Delta(\lambda, \a, \b, \r)$ or $\Theta(\lambda, \a, \b, \r)\, (2\b \notin \Z_+)$ constructed in Definition \ref{defi2.2}. Then $M(\lambda_1, \alpha_1, \beta_1, \gamma_1)\otimes V(\eta_1, \epsilon_1,\theta_1)$ and  $M(\lambda_2, \alpha_2, \beta_2, \gamma_2)\otimes V(\eta_2, \epsilon_2, \theta_2)$
are isomorphic as  $\mathfrak L$-modules if and only if
$$M(\lambda_1, \alpha_1, \beta_1, \gamma_1)\cong M(\lambda_2, \alpha_2, \beta_2, \gamma_2) \quad {\rm and }\quad V(\eta_1, \epsilon_1,\theta_1)\cong V(\eta_2, \epsilon_2, \theta_2).$$
\end{theo}
\begin{proof}
We only tackle  the case  $M(\lambda, \alpha, \beta, \gamma)=\Omega(\lambda,\a, \b, \r)$, since the other two cases can be treated similarly. The sufficiency is obvious and it suffices to show the necessity. Let $\phi$ be an
$\mathfrak L$-module isomorphism from $\Omega(\lambda_1,\a_1, \b_1, \r_1)\otimes V(\eta_1, \epsilon_1,\theta_1)$ to $\Omega(\lambda_2, \alpha_2, \beta_2, \gamma_2)\otimes V(\eta_2, \epsilon_2, \theta_2)$. Take a nonzero element $v\in V(\eta_1, \epsilon_1,\theta_1)$. Suppose
$$\phi(1\otimes v)=\sum_{i=0}^{n}a_i(t)s^{i}\otimes w_i,$$
where $a_i(t)\in\C[t], w_i\in V(\eta_2, \epsilon_2, \theta_2)$ with $a_{n}(t)\neq0, w_n\neq0$. There exists a positive integer $K={\rm max\,}\{K(v), K(w_i)\mid i=0,\cdots,n\}$ such that $d_m\cdot v=d_m\cdot w_i=e_m\cdot v=e_m\cdot w_i=f_m\cdot v=f_m\cdot w_i=h_m\cdot v=h_m\cdot w_i=0$ for all $m\geq K$ and $0\leq i \leq n$. For any two different integers $m_1, m_2\geq K$, we have
$$(\lambda_1^{-m_{1}}d_{m_{1}}-\lambda_1^{-m_{2}}d_{m_{2}})\cdot (1\otimes v)=(m_1-m_2)\gamma_1(1\otimes v).$$
Then by applying $\phi$ on both sides, we obtain
\begin{eqnarray}\label{for1}
&&(m_1-m_2)\gamma_1\sum_{i=0}^{n}a_i(t)s^{i}\otimes w_i\nonumber\\
&=&(\lambda_1^{-m_1}d_{m_1}-\lambda_1^{-m_2}d_{m_2})\cdot\sum_{i=0}^{n}a_i(t)s^{i}\otimes w_i\nonumber\\
&=&\sum_{i=0}^{n}(\frac{\lambda_2}{\lambda_1})^{m_1}(s+m_1\gamma_2)(s-m_1)^ia_i(t)\otimes w_i\nonumber\\
&&-\sum_{i=0}^{n}(\frac{\lambda_2}{\lambda_1})^{m_2}(s+m_2\gamma_2)(s-m_2)^ia_i(t)\otimes w_i.
\end{eqnarray}
Comparing the coefficients of $s^{n+1}$ in the above formula, we deduce that
$$\big((\frac{\lambda_2}{\lambda_1})^{m_1}-(\frac{\lambda_2}{\lambda_1})^{m_2}\big)(a_n(t)\otimes w_n)=0,$$
forcing $\lambda_1=\lambda_2$. Then \eqref{for1} can be simplified as
\begin{eqnarray}\label{for3}
&&(m_1-m_2)\gamma_1\sum_{i=0}^{n}a_i(t)s^{i}\otimes w_i\nonumber\\
&=&\sum_{i=0}^{n}\big((s+m_1\gamma_2)(s-m_1)^i-(s+m_2\gamma_2)(s-m_2)^i\big)a_i(t)\otimes w_i.
\end{eqnarray}
Regard it as a polynomial in $m_1, m_2$ with coefficients in $\Omega(\lambda_2, \alpha_2, \beta_2, \gamma_2)\otimes V(\eta_2, \epsilon_2, \theta_2)$. If $\gamma_2\neq0$, it follows from \eqref{for3} that $n=0$, since otherwise the coefficient $(-1)^n\gamma_2a_n(t)\otimes w_n$ of $m_1^{n+1}$ would be zero, yielding a contradiction. Then it follows from \eqref{for3} that $\gamma_1=\gamma_2\neq0$. If $\gamma_2=0$, \eqref{for3} becomes
\begin{equation*}\label{for2}
(m_1-m_2)\gamma_1\sum_{i=0}^{n}a_i(t)s^{i}\otimes w_i=s\sum_{i=0}^{n}\big((s-m_1)^i-(s-m_2)^i\big)a_i(t)\otimes w_i,
\end{equation*}
Also, by regarding it as a polynomial in $m_1, m_2$ with coefficients in $\Omega(\lambda_2, \alpha_2, \beta_2, \gamma_2)\otimes V(\eta_2, \epsilon_2, \theta_2)$, we see that $n\leq1$. If $n=1$, then the above formula becomes
\begin{equation*}
(m_1-m_2)\gamma_1\big(a_0(t)\otimes w_0+a_1(t)s\otimes w_1\big)=(m_2-m_1)a_1(t)s\otimes w_1,
\end{equation*}
yielding $\gamma_1=-1$ and $a_0(t)\otimes w_0=0$. So, $\phi(1\otimes v)=a_1(t)s\otimes w_1$. Since
\begin{equation}\label{for5}
\lambda^{-m}\alpha_{1}^{-1}e_{m}\cdot (1\otimes v)=1\otimes v \,\,{\rm for\,\,any}\,\, m\geq K,
\end{equation}
by applying $\phi$ on both sides of (\ref{for5}), it follows that
\begin{eqnarray*}\label{for4}
a_1(t)s\otimes w_1&=&\lambda^{-m}\alpha_{1}^{-1}e_{m}\cdot\big(a_1(t)s\otimes w_1\big)\nonumber\\
&=&\alpha_{1}^{-1}\alpha_{2}a_1(t-2)(s-m)\otimes w_1,
\end{eqnarray*}
forcing $\alpha_{1}^{-1}\alpha_{2}a_1(t-2)\otimes w_1=0$, a contradiction with the assumption that $a_1(t)s\otimes w_1\neq0$.
Therefore,   $n=0$ and $\gamma_1=0$. The preceding discussion shows that $\phi(1\otimes v)=a_0(t)\otimes w_0$ and $\gamma_1=\gamma_2$. Now applying $\phi$ to \eqref{for5} gives
$$\alpha_{1}^{-1}\alpha_{2}a_0(t-2)\otimes w_0=a_0(t)\otimes w_0,$$
which yields $a_0(t)\in\C^*$ and $\alpha_{1}=\alpha_{2}$. Without loss of generality, we assume that $a_0(t)=1$. Denote $\lambda=\lambda_1=\lambda_2$, $\alpha=\alpha_{1}=\alpha_{2}$ and  $\gamma=\gamma_1=\gamma_2$ in what follows. Thus there exists a linear injection $\tau:V(\eta_1, \epsilon_1, \theta_1)\rightarrow V(\eta_2, \epsilon_2, \theta_2)$ such that
\begin{equation}\label{equ0}\phi(1\otimes v)=1\otimes \tau(v), \,\, \forall\, v\in V(\eta_1, \epsilon_1, \theta_1).\end{equation}
For any $m\geq K$, the equations
\begin{eqnarray*}
&\phi(d_m\cdot (1\otimes v))=d_m\cdot \phi(1\otimes v), \\
&\phi(h_m\cdot (1\otimes v))=h_m\cdot \phi(1\otimes v),\\
&\phi(f_m\cdot (1\otimes v))=f_m\cdot \phi(1\otimes v)
\end{eqnarray*}
are respectively equivalent to
\begin{eqnarray*}
&\lambda^m\phi(s\otimes v)+\lambda^mm\gamma(1\otimes \tau(v))=\lambda^m(s\otimes \tau(v))
+\lambda^mm\gamma(1\otimes \tau(v)), \\
&\lambda^m\phi(t\otimes v)=\lambda^m(t\otimes \tau(v)),\\
&-\frac{\lambda^m}{\alpha}\phi\big((\frac{t^2}{4}+\frac{t}{2}-\beta_1(\beta_1+1))\otimes v\big)=-\frac{\lambda^m}{\alpha}(\frac{t^2}{4}+\frac{t}{2}-\beta_2(\beta_2+1))\otimes \tau(v),
\end{eqnarray*}
we see that
\begin{eqnarray}
\label{equ11}&\phi(s\otimes v)=s\otimes \tau(v), \\
\label{equ21}&\phi(t\otimes v)=t\otimes \tau(v),\\
\label{equ31}&\phi\big((\frac{t^2}{4}+\frac{t}{2}-\beta_1(\beta_1+1))\otimes v\big)=(\frac{t^2}{4}+\frac{t}{2}-\beta_2(\beta_2+1))\otimes \tau(v).
\end{eqnarray}
From \eqref{equ21} and
\begin{equation*}
\phi(h_m\cdot (t\otimes v))=h_m\cdot \phi(t\otimes v), \quad{\rm where}\,\,m\geq K,
\end{equation*}
we deduce that
\begin{equation}\label{equ4}
\phi(t^2\otimes v)=t^2\otimes \tau(v).
\end{equation}
This along with \eqref{equ21}-\eqref{equ31} gives $\beta_1=\beta_2$ or $\beta_1=-\beta_2-1$. Hence, $\Omega(\lambda_1,\a_1,\b_1,\r_1)\cong
\Omega(\lambda_2, \a_2,\b_2,\r_2)$ by Proposition \ref{pop1}.
Combining \eqref{equ0}-\eqref{equ21} with  \eqref{equ4}, we obtain
\begin{equation*}
\phi((X_m\cdot 1)\otimes v)=(X_m\cdot1)\otimes\tau(v), \quad{\rm where}\,\,X_m\in\{d_m, h_m, e_m, f_m\mid\forall m\in\Z\}.
\end{equation*}
This together with
$$\phi(X_m\cdot (1\otimes v))=X_m\cdot \phi(1\otimes v),\quad{\rm where}\,\,X_m\in\{d_m, h_m, e_m, f_m\mid\forall m\in\Z\}$$
gives
$$\phi(1\otimes (X_m\cdot v))=1\otimes (X_m\cdot \tau(v)),\quad{\rm where}\,\,X_m\in\{d_m, h_m, e_m, f_m\mid\forall m\in\Z\}.$$
Therefore,
$$\tau(X_m\cdot v)=X_m\cdot\tau(v),\quad{\rm where}\,\,X_m\in\{d_m, h_m, e_m, f_m\mid\forall m\in\Z\}, v\in V(\theta_1, \epsilon_1, \eta_1).$$
Since
$$\phi(C\cdot (1\otimes v))=C\cdot \phi (1\otimes v), \,\,\forall\,v\in V(\theta_1, \epsilon_1, \eta_1),$$  we see that $\tau(C\cdot v)=C\cdot \tau(v)$.
Thus, $\tau$ is a nonzero $\mathfrak L$-module homomorphism. Note that $V(\eta_1, \epsilon_1, \theta_1))$ and $V(\eta_2, \epsilon_2, \theta_2)$ are simple $\mathfrak L$-modules, $\tau$ is an  $\mathfrak L$-module isomorphism.  We complete the proof.
\end{proof}
\section{Comparison of tensor product modules with known non-weight modules}
In this section, we compare the tensor product modules constructed in the present paper with all other known non-weight $\mathfrak L$-modules, i.e., $U(\mathfrak{h})$-free modules  of rank one and Whittaker modules (cf. \cite{CH1}).

Let $\underline{\mu}=(\mu_1,\cdots,\mu_5)\in\C^5$. Assume that $J_{\underline{\mu}}$ is the left ideal of $U(\mathfrak{L}_{+})$ generated by $\{d_1-\mu_1, d_2-\mu_2, e_0-\mu_3,  f_1-\mu_4,C-\mu_5, d_j, e_k, f_l, h_m\mid j\geq 3, k\geq 1, l\geq 2, m\geq 1\}.$
Denote $N_{\underline{\mu}}:=U(\mathfrak{L}_{+})/J$. Then ${\rm Ind}(N_{\underline{\mu}}):=U(\mathfrak L)\otimes_{U(\mathfrak{L}_{+})} N_{\underline{\mu}}$ is a universal Whittaker module, and any Whittaker module is a quotient of some universal Whittaker module.

For any $r\in\Z_+, l, m\in\Z$, as in \cite{LLZ}, we denote
$$\omega_{l,m}^{(r)}=\sum_{i=0}^r\binom{r}{i}(-1)^{r-i}d_{l-m-i}d_{m+i}\in U(\mathfrak L).$$
\begin{lemm}\label{theo123456} Let $M(\lambda, \alpha, \beta, \gamma)$ and $V(\eta, \epsilon, \theta)$ be the $\mathfrak L$-modules defined as in \S \ref{pre}, and $\eta, \epsilon, \theta$ can not be identically zero. Then the following statements hold.
\begin{itemize}
\item[(1)] $d_i$ acts injectively on $M(\lambda, \alpha, \beta, \gamma)\otimes V(\eta, \epsilon, \theta)$ for any $i\in\mathbb{Z}$.
\item[(2)] For any $g(s,t)\in M(\lambda, \alpha, \beta, \gamma)$, we have
$$\omega_{l,m}^{(r)}(g(s,t))=0,\,\,\forall \,\,l, m, r\in\Z, \,\,r>2.$$
\item[(3)] For any $r>2$ and $0\neq g(s,t)\in M(\lambda, \alpha, \beta, \gamma)$, there exists $v\in V(\eta, \epsilon, \theta)$ and $l, m\in\Z$ such that
$$\omega_{l,m}^{(r)}(g(s,t)\otimes v)\neq0.$$
\end{itemize}
\end{lemm}
\begin{proof}

(1) For any $$0\neq v=\sum_{j=0}^m\sum_{k=0}^{n_j}s^jt^k\otimes v_{jk}\in M(\lambda, \alpha, \beta, \gamma)\otimes V(\eta, \epsilon, \theta)$$ with $v_{jk}\in V(\eta, \epsilon, \theta)$ for any $j,k$, and $v_{mn_m}\neq 0$, it follows from the $\mathfrak L$-module structure on $M(\lambda, \alpha, \beta, \gamma)$ in Definition \ref{defi2.2} that
$$d_i\cdot v=\sum_{k=0}^{n_m}\lambda^is^{m+1}t^k\otimes v_{mk}+\sum_{j=0}^{m}\sum_{k}s^jt^k\otimes w_{jk}\neq 0,$$
where $w_{jk}\in V(\eta, \epsilon, \theta)$ for any $j,k$. Hence, (1) follows.

(2) For any $h(t)\in\C[t]$ and $j\in\Z_+$, we have
\begin{eqnarray*}
\omega_{l,m}^{(r)}(s^jh(t))&=&\sum_{i=0}^r\binom{r}{i}(-1)^{r-i}d_{l-m-i}d_{m+i}(s^jh(t))\\
&=&\sum_{i=0}^r\binom{r}{i}(-1)^{r-i}d_{l-m-i}\big(\lambda^{m+i}(s+(m+i)\gamma)(s-m-i)^jh(t)\big)\\
&=&\sum_{i=0}^r\binom{r}{i}(-1)^{r-i}\lambda^{l}\big(s+(l-m-i)\gamma\big)\big(s+(m+i)\gamma-l+m+i\big)(s-l)^jh(t).
\end{eqnarray*}
This together with the following identity
$$\sum_{i=0}^r\binom{r}{i}(-1)^{r-i}i^j=0,\,\,\forall j, \,r\in\Z_+\,\,{\rm with}\,\,j<r$$
forces $\omega_{l,m}^{(r)}(g(s,t))=0$ provided $r>2$, proving (2).

(3) Fix any $r>2$. Take $v$ to be the highest weight vector of  $V(\eta, \epsilon, \theta)$. It is important to observe  that  the vectors  $d_{-2}v, d_{-3}v, \cdots, d_{-r-2}v$ are  linearly independent in $V(\eta, \epsilon, \theta)$. Take $l=r+1$ and $m=-r-2$. As $\omega_{l,m}^{(r)}(g(s,t))=0$ by (ii), we have
\begin{eqnarray*}
\omega_{l,m}^{(r)}(g(s,t)\otimes v)&=&\sum_{i=0}^r\binom{r}{i}(-1)^{r-i}d_{l-m-i}d_{m+i}(g(s,t)\otimes v)\\
&=&\sum_{i=0}^r\binom{r}{i}(-1)^{r-i}d_{l-m-i}(g(s,t))\otimes d_{m+i}(v)\\
&=&\sum_{i=0}^r\binom{r}{i}(-1)^{r-i}\lambda^{l-m-i}\big(s+(l-m-i)\gamma\big)g(s-l+m+i,t)\otimes d_{m+i}(v)\\
&\neq&0.
\end{eqnarray*}
Hence (3) follows.

\end{proof}

As a consequence of Lemma \ref{theo123456}, we have the following result which asserts that the tensor product modules constructed in the present paper are different from the other known non-weight modules.
\begin{prop}\label{theovvn}
Let  $\lambda,\a\in\C^*, \b,\r,\eta, \epsilon, \theta \in\C, $ and $\eta, \epsilon, \theta$ can not be identically zero. Then the tensor product module
$M(\lambda, \alpha, \beta, \gamma)\otimes V(\eta, \epsilon, \theta)$ is a new non-weight $\mathfrak L$-module.
\end{prop}
\begin{proof}
We need to show that the tensor product module $M(\lambda, \alpha, \beta, \gamma)\otimes V(\eta, \epsilon, \theta)$ is neither isomorphic to a Whittaker module nor isomorphic to a  $U(\mathfrak{h})$-free modules of rank one in \cite{CH1}. For that, let $W$ be a Whittaker module, then $W$ is isomorphic to a quotient of $N_{\underline{\mu}}$ for some $\underline{\mu}=(\mu_1,\cdots,\mu_5)\in\C^5$. It follows from Lemma \ref{theo123456} (i) that $M(\lambda, \alpha, \beta, \gamma)\otimes V(\eta, \epsilon, \theta)\ncong W$, since for any nonzero element $v\in{\rm Ind}(N_{\underline{\mu}})$ (resp. $w\in W$),  there exists a positive integer $i$ such that $d_i$ acts on $v$ (resp. $w$) trivially. Moreover, thanks to Lemma \ref{theo123456} (ii)  and (iii), $M(\lambda, \alpha, \beta, \gamma)\otimes V(\eta, \epsilon, \theta)\ncong M(\lambda', \alpha', \beta', \gamma')$ for any  $\lambda',\alpha'\in\C^*, \beta',\gamma'\in\C$. We complete the proof.
\end{proof}

\section{Realization of tensor product modules as induced modules}
In this section, we realize the tensor product modules $M(\lambda, \alpha, \beta, \gamma)\otimes \overline{V}(\eta, \epsilon, \theta)$ as certain induced
modules. For that, fix $\lambda\in\C^*$, let $$\mathfrak{b}_{\lambda}=\text{span}_{\C}\{d_m-\lambda^md_0,  f_m, h_n, e_n, C\mid m\in\N, n\in\Z_+\}$$ be the subalgebra of $\mathfrak L$.
\begin{defi}\label{defi2.4}\rm Let $\C[t]$ be the polynomial algebra in the variable $t$ with coefficients in $\C$. For $\lambda, \a\in\C^*, \b, \r, \eta,\epsilon,\theta\in\C$ and $g(t)\in\C[t]$, define the action of $\mathfrak{b}_{\lambda}$ on $\C[t]$ as follows:
\begin{align*}
\C[t]_{\lambda,\a, \b, \r, \eta,\epsilon,\theta}^{\Omega}:&\ \ \ \  (d_m-\lambda^md_0)\circ g(t)=\lambda^m (m\gamma-\eta) g(t),\\
& \ \ \ \ f_m\circ g(t)=-\frac{\lambda^m}\a(\frac{t}{2}-\b)(\frac{t}{2}+\b+1)g(t+2),\\
& \ \ \ \ h_n\circ g(t)=\lambda^n (t+\delta _{n,0}\epsilon) g(t),\ \ \ \ e_n\circ g(t)=\lambda^n \alpha g(t-2),\\
& \ \ \ \ C\circ g(t)=\theta g(t);\\
\C[t]_{\lambda,\a, \b, \r, \eta,\epsilon,\theta}^{\Delta}:&\ \ \ \ (d_m-\lambda^md_0)\circ g(t)=\lambda^m (m\gamma-\eta) g(t),\\
& \ \ \ \ f_m\circ g(t)=\lambda^m \alpha g(t+2),\\
& \ \ \ \ h_n\circ g(t)=\lambda^n (t+\delta _{n,0}\epsilon) g(t),\\
&\ \ \ \ e_n\circ g(t)=-\frac{\lambda^n}\a(\frac{t}{2}+\b)(\frac{t}{2}-\b-1)g(t-2),\\
& \ \ \ \ C\circ g(t)=\theta g(t);\\
\C[t]_{\lambda,\a, \b, \r, \eta,\epsilon,\theta}^{\Theta}:&\ \ \ \ (d_m-\lambda^md_0)\circ g(t)=\lambda^m (m\gamma-\eta) g(t),\\
& \ \ \ \ f_m\circ g(t)=-\frac{\lambda^m}\a(\frac{t}{2}-\b)g(t+2),\\
& \ \ \ \ h_n\circ g(t)=\lambda^n (t+\delta _{n,0}\epsilon) g(t),\ \ \ \ e_n\circ g(t)=\lambda^n \alpha(\frac{t}{2}+\b)\gamma g(t-2),\\
& \ \ \ \ C\circ g(t)=\theta g(t).
\end{align*}
where $ m\in\N, n\in\Z_+$.
\end{defi}

\begin{prop}
Keep notations as above, then $\C[t]_{\lambda,\a, \b, \r, \eta,\epsilon,\theta}^{\Omega}, \C[t]_{\lambda,\a, \b, \r, \eta, \epsilon,\theta}^{\Delta}$ and  $\C[t]_{\lambda,\a, \b, \r, \eta, \epsilon,\theta}^{\Theta}$ are $\mathfrak{b}_{\lambda}$-modules under the actions given in Definition \ref{defi2.4}.
\end{prop}
\begin{proof}
The assertion can be verified straightforwardly, we omit the details.
\end{proof}

\begin{remm}
Let $\lambda,\a\in\C^*,\b,\r,\eta,\epsilon,\theta\in\C$. Then by a similar argument as that in the proof in \cite[Proposition 2.5 ]{CH1}, the $\mathfrak{b}_{\lambda}$-modules $\C[t]_{\lambda,\a, \b, \r, \eta, \epsilon,\theta}^{\Omega}$ and $\C[t]_{\lambda,\a, \b, \r, \eta, \epsilon,\theta}^{\Delta}$ are always irreducible, while the $\mathfrak{b}_{\lambda}$-module $\C[t]_{\lambda,\a, \b, \r, \eta, \epsilon,\theta}^{\Theta}$ is irreducible if and only if $2\b \notin \Z_+$.
\end{remm}

We always assume that $\lambda,\a\in\C^*,\b,\r,\eta,\epsilon,\theta\in\C$ in the following. Now one can form the induced
$\mathfrak L$-modules as follows:
\begin{eqnarray*} &\text{Ind}(\C[t]_{\lambda,\a, \b, \r, \eta, \epsilon,\theta}^{\Omega}):=U(\mathfrak L)\otimes_{ U(\mathfrak{b}_{\lambda})}\C[t]_{\lambda,\a, \b, \r, \eta, \epsilon,\theta}^{\Omega};\\
&\text{Ind}(\C[t]_{\lambda,\a, \b, \r, \eta, \epsilon,\theta}^{\Delta}):=U(\mathfrak L)\otimes_{ U(\mathfrak{b}_{\lambda})}\C[t]_{\lambda,\a, \b, \r, \eta, \epsilon,\theta}^{\Delta};\\
&\text{Ind}(\C[t]_{\lambda,\a, \b, \r, \eta, \epsilon,\theta}^{\Theta}):=U(\mathfrak L)\otimes_{ U(\mathfrak{b}_{\lambda})}\C[t]_{\lambda,\a, \b, \r, \eta, \epsilon,\theta}^{\Theta}.\end{eqnarray*}
\begin{theo}\label{vec55}
Keep notations as above. Then as $\mathfrak L$-modules
\begin{eqnarray*} &\Omega(\lambda,\a,\b,\r)\otimes \overline{V}(\eta, \epsilon, \theta)\cong \text{Ind}(\C[t]_{\lambda,\a, \b, \r, \eta, \epsilon,\theta}^{\Omega});\nonumber\\
&\Delta(\lambda,\a,\b,\r)\otimes \overline{V}(\eta, \epsilon, \theta)\cong \text{Ind}(\C[t]_{\lambda,\a, \b, \r, \eta, \epsilon,\theta}^{\Delta});\nonumber\\
&\Theta(\lambda,\a,\b,\r)\otimes \overline{V}(\eta, \epsilon, \theta)\cong\text{Ind}(\C[t]_{\lambda,\a, \b, \r, \eta, \epsilon,\theta}^{\Theta}).\end{eqnarray*}
\end{theo}
\begin{proof}
In the following we only prove the first case, since a similar argument can be applied to the other two cases.

According to the $\mathrm{PBW}$ Theorem,  we see that
$\text{Ind}(\C[t]_{\lambda,\a, \b, \r, \eta, \epsilon,\theta}^{\Omega})$ has a basis
\begin{eqnarray*} &\mathcal{B}_1=\{f_{-q}^{F_{-q}}\cdots f_0^{F_0}e_{-p}^{E_{-p}}\cdots e_{-1}^{E_{-1}}h_{-m}^{H_{-m}}\cdots h_{-1}^{H_{-1}}d_{-n}^{D_{-n}}\cdots d_0^{D_0}\otimes t^i\mid D_0,\cdots,D_{-n},\\
& H_{-1}, \cdots, H_{-m}, E_{-1}, \cdots, E_{-p}, F_0,  \cdots, F_{-q}, i\in\Z_+\}
\end{eqnarray*}
and $\Omega(\lambda,\a,\b,\r)\otimes \overline{V}(\eta, \epsilon, \theta)$ has a basis
\begin{eqnarray*} &\mathcal{B}_2=\{t^is^{D_0}\otimes (f_{-q}^{F_{-q}}\cdots f_0^{F_0}e_{-p}^{E_{-p}}\cdots e_{-1}^{E_{-1}}h_{-m}^{H_{-m}}\cdots h_{-1}^{H_{-1}}d_{-n}^{D_{-n}}\cdots d_{-1}^{D_{-1}}\cdot v_h )\mid D_0,\cdots,D_{-n},\\
& H_{-1}, \cdots, H_{-m}, E_{-1}, \cdots, E_{-p}, F_0,  \cdots, F_{-q}, i\in\Z_+\}.
\end{eqnarray*}

Now we define the following linear map
\begin{equation*}
\phi: \text{Ind}(\C[t]_{\lambda,\a, \b, \r, \eta, \epsilon,\theta}^{\Omega}) \to\Omega(\lambda,\a,\b,\r)\otimes \overline{V}(\eta, \epsilon, \theta)\end{equation*}
by \begin{eqnarray*}
&&
\phi(f_{-q}^{F_{-q}}\cdots f_0^{F_0}e_{-p}^{E_{-p}}\cdots e_{-1}^{E_{-1}}h_{-m}^{H_{-m}}\cdots h_{-1}^{H_{-1}}d_{-n}^{D_{-n}}\cdots d_0^{D_0}\otimes t^i)\nonumber\\
&=&f_{-q}^{F_{-q}}\cdots f_0^{F_0}e_{-p}^{E_{-p}}\cdots e_{-1}^{E_{-1}}h_{-m}^{H_{-m}}\cdots h_{-1}^{H_{-1}}d_{-n}^{D_{-n}}\cdots d_0^{D_0}( t^i\otimes v_h )
\end{eqnarray*}
We claim that $\phi$ is an  $\mathfrak L$-module homomorphism. By Definition \ref{defi2.4}, one can observe that
\begin{equation*}(d_m-\lambda^md_0)(t^i\otimes v_h )=((d_m-\lambda^md_0)\circ t^i)\otimes v_h , \ \ \ \ \forall m\in\N.
\end{equation*}
Combining this with Definition \ref{defi2.4}, for any $i\in\Z_+$, $m\in\N, n\in\Z_+$ and $$x=f_{-q}^{F_{-q}}\cdots f_0^{F_0}e_{-p}^{E_{-p}}\cdots e_{-1}^{E_{-1}}h_{-m}^{H_{-m}}\cdots h_{-1}^{H_{-1}}d_{-n}^{D_{-n}}\cdots d_0^{D_0},$$  we have
\begin{eqnarray*}
\phi(xd_m\otimes t^i)&=&\phi(\lambda^mxd_0\otimes t^i+x(d_m-\lambda^md_0)\otimes t^i)\\
&=&\phi(\lambda^mxd_0\otimes t^i+x\otimes (d_m-\lambda^md_0)\circ t^i)\\
&=&\lambda^mxd_0(t^i \otimes v_h )+x((d_m-\lambda^md_0)\circ t^i\otimes v_h ),\\
&=&\lambda^mxd_0(t^i \otimes v_h )+x(d_m-\lambda^md_0)(t^i\otimes v_h ),\\
&=&xd_m(t^i\otimes v_h ),
\end{eqnarray*}
\begin{eqnarray*}
\phi(xf_m\otimes t^i)&=&\phi(x\otimes f_m\circ t^i)=x(f_m\circ t^i\otimes v_h)=xf_m(t^i\otimes v_h),
\end{eqnarray*}
\begin{eqnarray*}
\phi(xY_n\otimes t^i)&=&\phi(x\otimes Y_n\circ t^i)\\
&=&x(Y_n\circ t^i\otimes v_h)=xY_n(t^i\otimes v_h), \quad \text{where}\ Y_n=e_n \  \mathrm{or} \ h_n
\end{eqnarray*}
and
\begin{eqnarray*}
\phi(xC\otimes t^i)=\theta\phi(x\otimes t^i)=\theta x( t^i\otimes v_h)=xC(t^i\otimes v_h).
\end{eqnarray*}
Then for any $j\in\Z$, by the PBW Theorem, we can write
\begin{eqnarray*}
z_jx&=&\sum_{j_a}X_{j_a}x_{j_a}+\sum_{j_b}Y_{j_b}y_{j_b}d_{j_b}+\sum_{j_c}Z_{j_c} z_{j_c}h_{j_c}+\sum_{j_p}U_{j_p}u_{j_p}e_{j_p}\\
&&+\sum_{j_q}V_{j_q}v_{j_q}f_{j_q}+\sum_{j_l}W_{j_l}w_{j_l} C,
\end{eqnarray*}
where $z_j\in\{d_j, h_j, e_j, f_j\mid j\in\Z\}, X_{j_a}, Y_{j_b}, Z_{j_c}, U_{j_p}, V_{j_q}, W_{j_l}\in \C, x_{j_a}\otimes t^i, y_{j_b}\otimes t^i, z_{j_c}\otimes t^i, u_{j_p}\otimes t^i, v_{j_q}\otimes t^i, w_{j_l}\otimes t^i\in \mathcal{B}_1$ and $j_b, j_q\in\N, j_c, j_p\in\Z_+$. Now we deduce that for any $j\in\Z$,
\begin{eqnarray*}
\phi(z_jx\otimes t^i)&=&\phi\Big(\sum_{j_a}X_{j_a}x_{j_a}\otimes t^i+\sum_{j_b}Y_{j_b}y_{j_b}d_{j_b}\otimes t^i\\
&&+\sum_{j_c}Z_{j_c} z_{j_c}h_{j_c}\otimes t^i+\sum_{j_p}U_{j_p}u_{j_p}e_{j_p}\otimes t^i\\
&&+\sum_{j_q}V_{j_q}v_{j_q}f_{j_q}\otimes t^i+\sum_{j_l}W_{j_l}w_{j_l} C\otimes t^i\Big)\\
&=&\sum_{j_a}X_{j_a}x_{j_a}(t^i\otimes v_h)+\sum_{j_b}Y_{j_b}y_{j_b}d_{j_b}(t^i\otimes v_h)\\
&&+\sum_{j_c}Z_{j_c} z_{j_c}h_{j_c}(t^i\otimes v_h)+\sum_{j_p}U_{j_p}u_{j_p}e_{j_p}(t^i\otimes v_h)\\
&&+\sum_{j_q}V_{j_q}v_{j_q}f_{j_q}(t^i\otimes v_h)+\sum_{j_l}W_{j_l}w_{j_l}C(t^i\otimes v_h)\\
&=&z_jx(t^i\otimes v_h)=z_j\phi(x\otimes t^i)
\end{eqnarray*}
and
\begin{eqnarray*}
\phi(Cx\otimes t^i)=\theta\phi(x\otimes t^i)=\theta x(t^i\otimes v_h)=xC(t^i\otimes v_h)=Cx(t^i\otimes v_h)=C\phi(x\otimes t^i).
\end{eqnarray*}
This implies that $\phi$ is an  $\mathfrak L$-module homomorphism.

Next we shall show that $\phi$ is an  $\mathfrak L$-module isomorphism. Clearly, $t^i\otimes v_h\in \text{Im}(\phi)$ for all $i\in\Z_+$. Moreover, $d_0^j(t^i\otimes  v_h)\in \text{Im}(\phi)$, this implies that $t^is^j\otimes v_h\in \text{Im}(\phi)$ for all $j\in\Z_+$, i.e., $\Omega(\lambda,\a,\b,\r)\otimes v_h\subset \text{Im}(\phi)$. By applying the action of the elements of the form $f_{-q}^{F_{-q}}\cdots f_0^{F_0}e_{-p}^{E_{-p}}\cdots e_{-1}^{E_{-1}}h_{-m}^{H_{-m}}\cdots h_{-1}^{H_{-1}}d_{-n}^{D_{-n}}\cdots d_0^{D_0}$ on $\Omega(\lambda,\a,\b,\r)\otimes v_h$, we further deduce that
$\Omega(\lambda,\a,\b,\r)\otimes \overline{V}(\eta, \epsilon, \theta)\subset \text{Im}(\phi)$ and $\phi$ is surjective.
In order to obtain the injectivity of $\phi$, we first define a total order $``\prec"$ on $\mathcal{B}_2$
\begin{eqnarray*}
&&
t^is^{D_0}\otimes (f_{-q_{1}}^{F_{-q_{1}}}\cdots f_0^{F_0}e_{-p_{1}}^{E_{-p_{1}}}\cdots e_{-1}^{E_{-1}}h_{-m_{1}}^{H_{-m_{1}}}\cdots h_{-1}^{H_{-1}}d_{-n_{1}}^{D_{-n_{1}}}\cdots d_{-1}^{D_{-1}}\cdot v_h)\nonumber\\
&\prec&t^{i^{\prime}}s^{D_0^{\prime}}\otimes (f_{-q_{2}}^{F_{-q_{2}}^{\prime}}\cdots f_0^{F_0^{\prime}}e_{-p_{2}}^{E_{-p_{2}}^{\prime}}\cdots e_{-1}^{E_{-1}^{\prime}}h_{-m_{2}}^{H_{-m_{2}}^{\prime}}\cdots h_{-1}^{H_{-1}^{\prime}}d_{-n_{2}}^{D_{-n_{2}}^{\prime}}\cdots d_{-1}^{D_{-1}^{\prime}}\cdot v_h)
\end{eqnarray*}
if and only if
\begin{eqnarray*}
&
(D_{-1},\ldots ,D_{-n_{1}}, \overbrace{0, \ldots , 0}^{n_2}, H_{-1}, \ldots , H_{-m_{1}}, \overbrace{0, \ldots , 0}^{m_2},  E_{-1},  \ldots , E_{-p_{1}}, \overbrace{0, \ldots , 0}^{p_2},\nonumber\\
& F_0,  \ldots , F_{-q_{1}},\overbrace{0, \ldots , 0}^{q_2}, D_0, i)\nonumber\\
& <(D_{-1}^{\prime}, \ldots , D_{-n_{2}}^{\prime}, \overbrace{0, \ldots , 0}^{n_1}, H_{-1}^{\prime}, \ldots , H_{-m_{2}}^{\prime}, \overbrace{0, \ldots , 0}^{m_1}, E_{-1}^{\prime}, \ldots , E_{-p_{2}}^{\prime}, \overbrace{0, \ldots , 0}^{p_1},\nonumber\\
& F_0^{\prime}, \ldots , F_{-q_{2}}^{\prime}, \overbrace{0, \ldots , 0}^{q_1}, D_0^{\prime}, i^{\prime}),
\end{eqnarray*}
where
$$(a_1, \ldots, a_l)<(b_1, \ldots, b_l)\Longleftrightarrow \exists\, k>0  \mbox { \ such that \ } a_i=b_i \mbox { \ for all\ } i<k \mbox { \ and \ } a_k<b_k.$$
An explicit calculation gives
\begin{eqnarray*}
&&
f_{-q}^{F_{-q}}\cdots f_0^{F_0}e_{-p}^{E_{-p}}\cdots e_{-1}^{E_{-1}}h_{-m}^{H_{-m}}\cdots h_{-1}^{H_{-1}}d_{-n}^{D_{-n}}\cdots d_0^{D_0}( t^i\otimes v_h)\nonumber\\
&=&t^is^{D_0}\otimes (f_{-q}^{F_{-q}}\cdots f_0^{F_0}e_{-p}^{E_{-p}}\cdots e_{-1}^{E_{-1}}h_{-m}^{H_{-m}}\cdots h_{-1}^{H_{-1}}d_{-n}^{D_{-n}}\cdots d_{-1}^{D_{-1}}\cdot v_h)+\mbox{ \  lower terms. \ }
\end{eqnarray*}
Since $\mathcal{B}_2$ is a basis of $\Omega(\lambda,\a,\b,\r)\otimes \overline{V}(\eta, \epsilon, \theta)$, so is the set
\begin{eqnarray*} &\{f_{-q}^{F_{-q}}\cdots f_0^{F_0}e_{-p}^{E_{-p}}\cdots e_{-1}^{E_{-1}}h_{-m}^{H_{-m}}\cdots h_{-1}^{H_{-1}}d_{-n}^{D_{-n}}\cdots d_0^{D_0}( t^i\otimes v_h)\mid D_0,\cdots,D_{-n},\\
& H_{-1}, \cdots, H_{-m}, E_{-1}, \cdots, E_{-p}, F_0,  \cdots, F_{-q}, i\in\Z_+\}.
\end{eqnarray*}
Thus, $\phi$ is injective, completing the proof.
\end{proof}

As a consequence of Theorem \ref{theoo1}, Theorem \ref{vec55}, Proposition \ref{pop1} and \cite[Corollary 1]{K}, we have the following sufficient and necessary conditions for reducibility of the induced modules

\begin{coro}\label{last cor}
Let $\lambda,\a\in\C^*,\b,\r,\eta,\epsilon,\theta\in\C$. Then the following statements hold.
\begin{itemize}
\item[(1)] The induced module $\text{\rm Ind}(\C[t]_{\lambda,\a, \b, \r, \eta, \epsilon,\theta}^{M})$ for $M\in\{\Omega,\Delta\}$ is reducible if and only if one of the following conditions holds.
\begin{itemize}
\item[(i)] $\epsilon+m(\theta+2)-n=0$ for some $m, n\in\Z_+$.
\item[(ii)] $-\epsilon+m(\theta+2)-n-1=0$ for some $m, n\in\N$.
\item[(iii)] $\theta\neq-2$ and there exists some $m, n\in\N$ such that
\begin{eqnarray*}
&&\eta+\frac{\epsilon^2+2\epsilon}{4(\theta+2)}+\frac{1}{48}\Big(\frac{-\theta^2+14\theta+26}{\theta+2}(m^2+n^2)\\
&&\,\,+\sqrt{\frac{(\theta^2-2\theta-2)(\theta^2-26\theta-50)}{(\theta+2)^2}}(m^2-n^2)
-24mn+\frac{2\theta^2-4\theta-4}{\theta+2}\Big)=0.
\end{eqnarray*}
\end{itemize}

\item[(2)]  The induced module $\text{\rm Ind}(\C[t]_{\lambda,\a, \b, \r, \eta, \epsilon,\theta}^{\Theta})$ is reducible if and only if $2\b\in\Z_+$ or one of the conditions {\rm (i), (ii), (iii)}  in {\rm (1)} holds.
\end{itemize}
\end{coro}

\begin{proof}
It follows from  Theorem \ref{theoo1} and  Theorem \ref{vec55} that $\text{Ind}(\C[t]_{\lambda,\a, \b, \r, \eta, \epsilon,\theta}^{M})$ is irreducible for $M\in\{\Omega,\Delta, \Theta\}$ if and only if both $M(\lambda,\a,\b,\r)$ and $\overline{V}(\eta, \epsilon, \theta)$ are irreducible. Consequently, the desired assertion follows directly from Proposition \ref{pop1} and \cite[Corollary 1]{K}.
\end{proof}

\subsection*{Acknowledgements}
Yu-Feng Yao is grateful to professor Yufeng Pei for providing the references \cite{K} and \cite{LPX}.


\begin{thebibliography}{9999}

\bibitem{BM} P. Batra,  V. Mazorchuk, Blocks and modules for Whittaker pairs, {\it J. Pure Appl. Algebra} {\bf 215}, 1552--1568 (2011).

\bibitem{B} Y. Billig,  A category of modules for the full toroidal Lie algebra, {\it International Mathematics Research Notices}, 46pp (2006).


\bibitem{CTZ} Y.  Cai, H. Tan, K. Zhao, Module structure on $U(\mathfrak{h})$ for Kac-Moody algebras (in Chinese), {\it Sci. Sin. Math.} {\bf 47}(11), 1491--1514 (2017).

\bibitem{CZ} Y. Cai, K. Zhao, Module structure on $U(\mathfrak{h})$ for basic Lie superalgebras, {\it Toyama Math. J.}  {\bf 37}, 55--72 (2015).




\bibitem{CG2} H. Chen,  X. Guo, Tensor product weight modules over the Virasoro algebra, {\it J. Lond. Math. Soc.} {\bf 88}, 829--844 (2013).

\bibitem{CHSY} H. Chen,  J. Han, Y. Su, X. Yue, Two classes of non-weight modules over the twisted Heisenberg-Virasoro algebra, {\it Manuscripta. Math.} {\bf 160}(1-2), 265--284 (2019).


\bibitem{CH1} Q. Chen,  J. Han, Non-weight modules over the affine-Virasoro algebra of type $A_1$,  {\it J. Math. Phys.} {\bf 60}, 071707 (2019).



\bibitem{CL} S. Cheng,  N. Lam, Finite conformal modules over the $N=2,3,4$ superconformal algebras,  {\it J. Math. Phys}. {\bf 42}(2), 906--933 (2001).



\bibitem{EJ} S. Eswara Rao, C. Jiang, Classification of irreducible integrable representations for the full toroidal Lie algebras, {\it J. Pure Appl. Algebra} {\bf 200}, 71--85 (2005).


\bibitem{GHL} Y. Gao, N. Hu, D. Liu, Representations of the affine-Virasoro algebra of type $A_1$,
{\it J. Geom. Phys.} {\bf 106}, 102--107  (2016).



\bibitem{GLW} X. Guo, X. Liu, J. Wang, New irreducible tensor product modules for the Virasoro algebra, {\it Asian J. Math.} {\bf 24(2)}, 191--206 (2020).

\bibitem{HCS} J. Han, Q. Chen,  Y. Su, Modules over the algebras $\mathcal{V}ir(a,b)$,  {\it Linear Algebra Appl.} {\bf 515}, 11--23 (2017).




\bibitem{JY} C. Jiang, H. You,  Irreducible representations for the affine-Virasoro Lie algebras of type $B_{l}$,  {\it Chinese Ann. Math. Ser. B} {\bf  25}(3), 359--368 (2004).

\bibitem{K} V. Kac, Highest weight representations of conformal current algebras, {\it Symposium on Topological and Geometric Methods in Field Theory$,$ Espoo$,$ Finland$,$ World Scientific}, 3--16 (1986).


\bibitem{K1} G. Kuroki, Fock space representations of an affine Lie algebras and integral representations in the Wess-Zumino-Witten models, {\it Comm. Math. Phys.} {\bf 142}(3), 511--542 (1991).

\bibitem{LLZ} G. Liu,  R. L\"{u},  K.  Zhao, A class of simple weight Virasoro modules, {\it J. Algebra} {\bf 424}, 506--521 (2015).

\bibitem{LPX} D. Liu, Y. Pei,  L. Xia, Classification of quasi-finite irreducible modules over affine-Virasoro algebras, {\it J. Lie  Theory}, to appear, (2021).


\bibitem{LQ} X. Liu,  M.  Qian, Bosonic Fock representations of the affine-Virasoro algebra, {\it J. Phys. A} {\bf 27}(5), 131--136 (1994).







%


\bibitem{MW}  V. Mazorchuk, E. Weisner, Simple Virasoro modules induced from codimension one subalgebras of the positive part, {\it  Proc. Amer. Math. Soc}. {\bf 142}(11), 3695--3703 (2012).

\bibitem{MZ} V. Mazorchuk, K. Zhao, Simple Virasoro modules which are locally finite over a positive part, {\it Selecta Math. (N.S.)} {\bf 20}(3), 839--854 (2014).

\bibitem{N} J. Nilsson, Simple $\mathfrak{sl_{n+1}}$-module structures on  $U(\mathfrak{h})$, {\it J. Algebra} {\bf 424}, 294--329 (2015).

\bibitem{N1} J. Nilsson, $U(\mathfrak{h})$-free modules and coherent families, {\it J. Pure Appl. Algebra}  {\bf 220},  1475--1488 (2016).



\bibitem{TZ1} H. Tan, K. Zhao, Irreducible Virasoro modules from tensor products, {\it Ark. Mat.} {\bf 54}, 181--200 (2016).

\bibitem{TZ} H. Tan,  K.  Zhao, $\mathfrak{W_n^+}$ and $\mathfrak{W_n}$-module structures on  $U(\mathfrak{h_n})$, {\it J. Algebra} {\bf 424}, 257--375 (2015).



\bibitem{XH} L. Xia, N. Hu, Irreducible representations for Virasoro-toroidal Lie algebras, {\it J. Pure Appl. Algebra}  {\bf 194},  213--237 (2004).

\bibitem{Z} H. Zhang, A class of representations over the Virasoro algebra, {\it J. Algebra}  {\bf 190},  1--10 (1997).


\end{thebibliography}
\end{document}